\documentclass{article}
\usepackage[T1]{fontenc}
\usepackage{amsmath,amssymb,amsthm,calc,graphicx,pdfpages,bm,tikz,mathpazo,txfonts}
\usepackage[left=2.5cm, top=2.5cm,bottom=2.5cm,right=2.5cm]{geometry}

\newtheorem {theorem}{Theorem} 
\newtheorem {proposition}[theorem]{Proposition}
\newtheorem {lemma}[theorem]{Lemma}

{\theoremstyle{definition}

}
{\theoremstyle{theorem}
\newtheorem {remark}[theorem]{Remark}

}

\newcommand{\Var}{\operatorname{Var}}

\def\ba{\begin{array}}
\def\ea{\end{array}}
\def\bea{\begin{eqnarray} \label}
\def\eea{\end{eqnarray}}
\def\be{\begin{equation} \label}
\def\ee{\end{equation}}
\def\bit{\begin{itemize}}
\def\eit{\end{itemize}}
\def\ben{\begin{enumerate}}
\def\een{\end{enumerate}}

\def\BB{\varmathbb{B}}

\def\EE{\varmathbb{E}}

\def\NN{\varmathbb{N}}
\def\PP{\varmathbb{P}}

\def\RR{\varmathbb{R}}

\def\g{\gamma}
\def\d{\delta}

\def\l{\lambda}

\def\G{\Gamma}

\def\L{\Lambda}

\def\cB{\mathcal{B}}

\def\cE{\mathcal{E}}
\def\cF{\mathcal{F}}

\def\dint{\textup{d}}

\begin{document}

\title{\bfseries Gaussian fluctuations for edge counts in\\ high-dimensional random geometric graphs}

\author{Jens Grygierek\footnotemark[1]\,\, and Christoph Th\"ale\footnotemark[2]}

\date{}
\renewcommand{\thefootnote}{\fnsymbol{footnote}}
\footnotetext[1]{Institute of Mathematics, Osnabr\"uck University, Germany. Email: jens.grygierek@uni-osnabrueck.de}

\footnotetext[2]{
Faculty of Mathematics, Ruhr University Bochum, Germany. Email: christoph.thaele@rub.de}

\maketitle

\begin{abstract}
\noindent Consider a stationary Poisson point process in $\RR^d$ and connect any two points whenever their distance is less than or equal to a prescribed distance parameter. This construction gives rise to the well known random geometric graph. The number of edges of this graph is counted that have midpoint in the $d$-dimensional unit ball. A quantitative central limit theorem for this counting statistic is derived, as the space dimension $d$ and the intensity of the Poisson point process tend to infinity simultaneously.
\end{abstract}

\section{Introduction and main result}\label{sec:ïntro}

Fix an intensity $\l\in(0,\infty)$ and a distance parameter $\d\in(0,\infty)$, and let $\eta_\l$ be a stationary Poisson point process in $\RR^d$, $d\in\NN$, with intensity $\l$. The points of $\eta_\l$ are taken as the vertices of a random graph and we connect any two distinct vertices by an edge provided that their distance is less than or equal to $\d$. By this construction the random geometric graph in $\RR^d$ arises. Random graphs of this type have received considerable attention and by now belong to the heart of geometric probability and stochastic geometry. For an historical account and for more background material we refer the reader to the research monograph \cite{PenroseBook} of Penrose.

In this paper we are interested in the number of edges of the random geometric graph that have their midpoint in the $d$-dimensional unit ball $\BB^d$, that is, in the edge counting statistic
\begin{equation}
\cE(\l,\d,d):=\frac{1}{2}\sum_{\substack{x,y\in\eta_\l\\ x\neq y}}{\bf 1}\Big\{\|x-y\|\leq\d,\frac{x+y}{2}\in\BB^d\Big\}\,,
\end{equation}
where $\|x-y\|$ denotes the Euclidean distance of $x$ and $y$. We investigate the asymptotic distributional behaviour of $\cE(\l,\d,d)$, as $\delta\to 0$ and the intensity as well as the space dimension $d$ tend to infinity simultaneously. This set-up is opposed to most of the existing literature in which the focus lies on random geometric graphs in $\RR^d$ with some fixed space dimension $d$ (we refer to paper of Bubeck, Ding, Eldan and R\'acz \cite{BubeckDingEldanRacz} and that of Devroye, Gy\"orgy, Lugosi and Udina \cite{DevroyeGyorgyLugosiUdina} for notable exceptions, where, however, questions concerning the high-dimensional fluctuations are not touched). However, in view of the strong recent interest in the statistics of high-dimensional data sets and given the application of random geometric graphs to cluster analysis, we believe that it is worth investigating the probability theory behind high-dimensional random geometric graphs. In addition, it is the purpose of the present text to demonstrate that the Malliavin-Stein approach for Poisson functionals, which has found considerable attention in stochastic geometry over the last years (we refer to \cite{PeccatiReitzner} for a recent overview), can successfully be applied also to spatial random models in high dimensions.

To present our result, let us introduce the following more specialized set-up. At first, we choose a dimension-dependent distance parameter $\d=\d_d$, namely we take
\begin{equation}\label{eq:DefDeltad}
\d_d=\frac{1}{d}\,,
\end{equation}
which implies that $\d_d\in(0,1)$ for all $d \geq 2$ (the motivation for our choice is explained in Remark \ref{remark:choose_delta} below). We notice that $\d_d\to 0$, as $d\to\infty$. Next, we choose -- implicitly -- a dimension-dependent intensity  $\l=\l_d$ by requiring that
\begin{equation}\label{eq:DefLambdad}
\lim_{d\to\infty}\kappa_d^2\l_d^2\d_d^d = \infty\,.
\end{equation}
Here and below, $\kappa_d:=V_d(\BB^d)$ denotes the volume of the $d$-dimensional unit ball. Roughly speaking, the growth of the intensity parameter has to compensate the exponential decay of $\kappa_d$, which behaves like $\frac{1}{\sqrt{\pi d}}\Big(\frac{2\pi e}{d}\Big)^{d/2}$, as $d\to\infty$, according to Stirling's formula. This means that, as $d \rightarrow \infty$, the intensity $\lambda_d$ has to grow to infinity faster than $\sqrt{\pi d} \Big(\frac{d}{2 \pi e}\Big)^{d/2} d^d$.
To simplify our notation we shall use the abbreviation  $\cE_d$ for $\cE(\l_d,\d_d,d)$ with $\d_d$ and $\l_d$ given by \eqref{eq:DefDeltad} and \eqref{eq:DefLambdad}, respectively. 
We can now formulate the first results dealing with the first- and the second-order moment of the random variables $\cE_d$:
\begin{equation}\label{eq:Expectation}
\EE[\cE_d] = \frac{1}{2}\,\kappa_d^2\,\l_d^2\,\d_d^d
\end{equation}
and
\begin{equation}\label{eq:Variance}
\frac{1}{2} \kappa_d^2 \lambda_d^2 \delta_d^d + \Big( 1 - \tfrac{\delta_d}{2} \Big)^d \kappa_d^3 \lambda_d^3 \delta_d^{2d} \leq \Var[\cE_d] \leq \frac{1}{2} \kappa_d^2 \lambda_d^2 \delta_d^d + \Big( 1 + \tfrac{\delta_d}{2} \Big)^d \kappa_d^3 \lambda_d^3 \delta_d^{2d}
\end{equation}
see Lemma \ref{lem:VarianceBound} below. We remark that the proof of \eqref{eq:Expectation} and \eqref{eq:Variance} is based on a multiple use of Mecke's formula for Poisson point processes that we re-phrase below, see \eqref{eq:Mecke}.

We turn now to our main result, that is, the quantitative central limit problem for the edge counting statistics $\cE_d$, as the space dimension $d$ tends to infinity. The rate of convergence in this limit theorem will be measured by the so-called Wasserstein distance $d_W(\,\cdot\,,\,\cdot\,)$ (see Section \ref{sec:2ndOrderPoincare} below for a formal definition). Finally, we shall indicate convergence in distribution by writing $\overset{D}{\longrightarrow}$.

\begin{theorem}\label{thm:CLT}
Let $Z$ be a standard Gaussian random variable. Then one can find absolute constants $c_1,c_2\in(0,\infty)$ such that
$$
d_W\Bigg(\frac{\cE_d-\EE[\cE_d]}{\sqrt{\Var[\cE_d]}},Z\Bigg) \leq c_1\,(\kappa_d\l_d)^{-\frac{1}{2}}\,\max\Big\{1,(\kappa_d\,\l_d\,\d_d^d)^{-\frac{1}{2}}\Big\}\,,
$$
whenever $d\geq c_2$. In particular, one has that
$$
\frac{\cE_d-\EE[\cE_d]}{\sqrt{\Var[\cE_d]}} \overset{D}{\longrightarrow} Z\,,\qquad\text{as}\qquad d\to\infty\,.
$$
\end{theorem}

The central limit theorem and especially the bound for Wasserstein distance in the previous theorem shares some similarities with the central limit theorem for the edge counting statistic in \cite{ReitznerSchulteThaele}. However, while in the latter result the space dimension is fixed, it increases in our set-up. In this context, all constants arising in \cite{ReitznerSchulteThaele} and which also arise in our approach need to be treated carefully in order to distinguish their dimension dependent behaviour from absolute constants. Moreover, while the proof of the central limit theorem in \cite{ReitznerSchulteThaele} is based on a general central limit theorem for Poisson $U$-statistics from \cite{ReitznerSchulte}, the proof we present is slightly different and uses the second-order Poincaré inequality for Poisson functionals from the recent paper \cite{LastPeccatiSchulte} of Last, Peccati and Schulte.


\medbreak

The rest of this text is structured as follows. In Section \ref{sec:Preliminaries} we recall some necessary background material and, in particular, re-phrase there the second-order Poincar\'e inequality for Poisson functionals. A bound for the normal approximation of second-order $U$-statistics will be derived in Section \ref{sec:2ndOrderUstatistics}, while the final Section \ref{sec:Proof} contains the proof of Theorem \ref{thm:CLT}.

\section{Preliminaries}\label{sec:Preliminaries}

\subsection{Notation}

The $d$-dimensional Euclidean space is denoted by $\RR^d$ and we let $\cB^d$ be the Borel $\sigma$-field on $\RR^d$. The Lebesgue measure on $\RR^d$ is indicated by $V_d$. A $d$-dimensional ball with radius $r > 0$ and centre in $z \in \RR^d$ is defined by
\begin{align*}
\BB^d_{r}(z) := \{ x \in \RR^d : \| x - z \| \leq r\}\,,
\end{align*}
where $\|\,\cdot\,\|$ stands for the usual Euclidean distance. We shall write $\BB^d$ instead of $\BB^d_1(0)$ and denote by $\kappa_d := V_d(\BB^d) =\pi^\frac{d}{2}/\Gamma[1+d/2]$ the volume of the $d$-dimensional unit ball $\BB^d$, where $\G[\,\cdot\,]$ is Euler's gamma function.

We use the symbol ${\sf N}$ to indicate the class of counting measures on $\RR^d$ and supply the space $\sf N$ with the smallest $\sigma$-field ${\cal N}_\sigma$ such that all mappings of the form $\mu\mapsto\mu(B)$ with $\mu\in{\sf N}$ and $B\in\cB^d$ are measurable. It will be convenient for us to identify a counting measure $\mu\in{\sf N}$ with its support and to write $x\in\mu$ if the point $x\in\RR^d$ is charged by $\mu$. The Dirac measure concentrated at a point $x\in\RR^d$ is denoted by $\d_x$.

\subsection{Poisson functionals, Mecke's formula and a second-order Poincar\'e inequality}\label{sec:2ndOrderPoincare}

We let $(\Omega,\cF,\PP)$ be our underlying probability space. A Poisson point process $\eta$ on $\RR^d$ with intensity measure $\L$ is a random counting measure on $\RR^d$, that is, a random element in $\sf N$, such that (i) $\eta(B)$ is Poisson distributed with mean $\L(B)$ for all $B\in\cB^d$ and (ii) $\eta(B_1),\ldots,\eta(B_m)$ are independent random variables whenever the sets $B_1,\ldots,B_m\in\cB^d$, $m\in\NN$, are pairwise disjoint. A Poisson point process is called stationary if its intensity measure $\L$ is a constant multiple $\l\geq 0$ of the Lebesgue measure on $\RR^d$. The constant $\l$ is called the intensity of the Poisson point process and we will always assume that $\l\in(0,\infty)$. It is well known that such a Poisson point process $\eta$ satisfies the following multivariate Mecke formula, see \cite[Theorem 4.4]{LastPenrose}. For integers $m\in\NN$ and non-negative measurable functions $g:(\RR^d)^m\times{\sf N}\to\RR$ it says that
\begin{equation}\label{eq:Mecke}
\begin{split}
&\EE\sum_{(x_1,\ldots,x_m)\in\eta^m_{\neq}}h(x_1,\ldots,x_m;\eta)\\
&\qquad\qquad= \l^m\int_{(\RR^d)^m}\EE[h(x_1,\ldots,x_m;\eta+\d_{x_1}+\ldots+\d_{x_m})]\,\dint(x_1,\ldots,x_m)\,,
\end{split}
\end{equation}
where $\eta^m_{\neq}$ is the collection of $m$-tuples of pairwise distinct points charged by $\eta$ and $\EE$ stands for the expectation (integration) with respect to our probability measure $\PP$.

By a Poisson functional $F$ we understand a random variable that is almost surely of the form $F=f(\eta)$, where $f:{\sf N}\to\RR$ is some measurable function, the so-called representative of $F$. For a Poisson functional $F$ with representative $f$ and $x\in\RR^d$ we define the difference operator $D_xF$ of $F$ as follows:
$$
D_xF:=f(\eta+\d_x)-f(\eta)\,.
$$
Furthermore, for two points $x_1,x_2\in\RR^d$ the second-order difference operator $D_{x_1,x_2}F$ applied to $F$ is given by
\begin{align*}
D_{x_1,x_2}F &:=D_{x_2}(D_{x_1}F)=D_{x_1}(D_{x_2}F)\\
&=f(\eta+\d_{x_1}+\d_{x_2})-f(\eta+\d_{x_1})-f(\eta+\d_{x_2})+f(\eta)\,.
\end{align*}

The first- and second-order difference operators can be used to reformulate a bound for the `distance' between a Poisson functional $F$ and a standard Gaussian random variable $Z$. To measure the closeness of $F$ and $Z$ we will use the Wasserstein distance $d_W(F,Z)$, which is given by
$$
d_W(F,Z):=\sup_{h\in{\rm Lip}(1)}\big|\EE[h(F)]-\EE[h(Z)]\big|\,.
$$
Here, the supremum runs over the class ${\rm Lip}(1)$ of Lipschitz functions $h:\RR\to\RR$ with Lipschitz constant less than or equal to $1$. In the definition we implicitly assume that $F$ and $Z$ are both defined on our probability space $(\Omega,\cF,\PP)$. We point out that convergence in Wasserstein distance of a sequence of random variables implies convergence in distribution.

We are now prepared to rephrase a version of the main result from \cite{LastPeccatiSchulte}, a so-called second-order Poincar\'e inequality for Poisson functionals, see also \cite[Theorem 21.3]{LastPenrose}. It is the main device in our proof of Theorem \ref{thm:CLT}.

\begin{proposition}\label{prop:SecondOrderPoincare}
Let $\eta$ be a Poisson point process on $\RR^d$ with intensity measure $\L$ and let $F$ be a Poisson functional satisfying $\EE[F]=0$, $\Var[F]=1$ and $\EE\int_{\RR^d}(D_xF)^2\,\L(\dint x)<\infty$. Further, let $Z$ be a standard Gaussian random variable. Defining
\begin{align*}
\g_1(F) &:=\int_{(\RR^d)^3}\big(\EE[(D_{x_1}F)^4]\,\EE[(D_{x_2}F)^4]\,\EE[(D_{x_1,x_3}F)^4]\,\EE[(D_{x_2,x_3}F)^4]\big)^{1/4}\,\L^3(\dint(x_1,x_2,x_3))\,,\\
\g_2(F) &:= \int_{(\RR^d)^3}\big(\EE[(D_{x_1,x_3}F)^4]\,\EE[(D_{x_2,x_3}F)^4]\big)^{1/2}\,\L^3(\dint(x_1,x_2,x_3))\,,\\
\g_3(F) &:= \int_{\RR^d}\EE[|D_xF|^3]\,\L(\dint x)\,,
\end{align*}
one has that
$$
d_W(F,Z) \leq 2\sqrt{\g_1(F)} + \sqrt{\g_2(F)}+\g_3(F)\,.
$$
\end{proposition}
\begin{proof}
This is a direct consequence of Theorem 1.1 in \cite{LastPeccatiSchulte} and the Cauchy-Schwarz inequality.
\end{proof}

We remark that a similar but more involved bound also exists for the so-called Kolmogorov distance $d_K(F,Z):=\sup_{t\in\RR}|\PP(F\leq t)-\PP(Z\leq t)|$. We have decided to restrict to the Wasserstein distance in order to keep the presentation transparent and to focus on the principal mathematical ideas.

%
%
%

\section{A general bound for second-order $U$-statistics}\label{sec:2ndOrderUstatistics}
The purpose of the present section is to provide a general bound for the normal approximation of second-order $U$-statistics in the sense of \cite{ReitznerSchulte} with non-negative kernels based on a Poisson point process $\eta$ in $\RR^d$ having intensity measure $\L$. Formally, we define
	\begin{align}\label{eq:generalF}
		F_d := \frac{1}{2} \sum\limits_{(y_1,y_2) \in {\eta}^2_{\neq}} h(y_1,y_2)
	\end{align}
and assume that $h:\RR^d\times\RR^d\to[0,\infty)$ is a symmetric measurable function, which we allow to depend on the space dimension $d$. Furthermore, we assume that $\EE[F_d^2]<\infty$. Then, by Mecke's formula \eqref{eq:Mecke}, we have that
\begin{equation}\label{eq:ExpectationUstatistic}
\EE[F_d] = \int_{\RR^d}\int_{\RR^d}h(x_1,x_2)\,\L(\dint x_1)\L(\dint x_2)
\end{equation}
and
\begin{equation}\label{eq:UStatisticVariance}
\begin{split}
\Var[F_d] &= \int_{\RR^d}\Big(\int_{\RR^d}h(x_1,x_2)\,\L(\dint x_2)\Big)^2\L(\dint x_1)\\
&\qquad\qquad\qquad+\frac{1}{2}\int_{\RR^d}\int_{\RR^d}h(x_1,x_2)^2\,\L(\dint x_2)\L(\dint x_1)\,,
\end{split}
\end{equation}
see also \cite{LastPenrose}. We further denote $\sigma^2 := \Var[F_d]$ and put
	\begin{align*}
		\widetilde{F_d} := \frac{F_d - \EE[F_d]}{\sigma}\,.
	\end{align*}		
Finally, for $k,\ell\in\NN$ we define the two parameter integrals
	\begin{align*}
		A_k(x) & := \int_{\RR^d} h^k(x,y) \, \L(\dint y)\,,\qquad x\in\RR^d\,,\\
		B_{k,\ell}(x_1,x_2) & := \int_{\RR^d} h^k(x_1,y) \, h^\ell(x_2,y) \, \L(\dint y)\,,\qquad x_1,x_2\in\RR^d\,.
	\end{align*}
We will see in Proposition \ref{prop:generalCLT} below that $\g_1(F_d),\g_2(F_d)$ and $\g_3(F_d)$ defined in Proposition \ref{prop:SecondOrderPoincare} can be expressed in terms of $A_k$ and $B_{k,\ell}$ for special choices of $k$ and $\ell$.

According to Lemma 3.3 in \cite{ReitznerSchulte} and the assumed symmetry of $h$ it is clear that
	\begin{align}\label{eq:DxFh}
		D_xF_d = \sum\limits_{y \in \eta_d} h(y,x)\qquad\text{and}\qquad D_{x_1,x_2}F_d = h(x_1,x_2)
	\end{align}
for all $x,x_1,x_2\in\RR^d$. Moreover, from the definition of the difference operator it follows that
	\begin{align}\label{eq:DxFhs}
		D_x \widetilde{F_d} = \frac{D_xF_d}{\sigma}\qquad\text{and}\qquad D_{x_1,x_2} \widetilde{F_d} = \frac{D_{x_1,x_2}F_d}{\sigma}\,.
	\end{align}

Next, we compute the expectations occurring at the right-hand side of Proposition \ref{prop:SecondOrderPoincare} to prepare the bounds for the three terms $\gamma_1(\widetilde{F_d})$, $\gamma_2(\widetilde{F_d})$ and $\gamma_3(\widetilde{F_d})$. 

\begin{lemma}\label{lemma:gamma}
Let $x,x_1,x_2\in\RR^d$.  Then
\begin{itemize}
\item[(a)] $\EE[|D_xF_d|^3] = A_1(x)^3 + 3 A_2(x)A_1(x) + A_3(x)$,
\item[(b)] $\EE[(D_xF_d)^4] = P(x)$ with
	\begin{align}\label{eq:poly:P}
		P(x) := A_1(x)^4 + 6 A_2(x)A_1(x)^2 + 3 A_2(x)^2 + 4 A_3(x)A_1(x) + A_4(x)
	\end{align}
	and
\item[(c)] $\EE[(D_{x_1,x_2}F_d)^4] = h(x_1, x_2)^4$.
\end{itemize}
	\begin{proof}
	Since the function $h$ is assumed to be non-negative, we have that
		\begin{align}\label{eq:DxF3:1}
			\EE[|D_xF_d|^3] =\EE[(D_xF_d)^3] = \EE \sum\limits_{(y_1,y_2,y_3) \in \eta^3}h(y_1,x)h(y_2,x)h(y_3,x)
		\end{align}
		for all $x\in\RR^d$.
	Splitting the sum and using the symmetry of $h$ as well as Mecke's formula \eqref{eq:Mecke} for each summand leads to 
		\begin{align*}
			& \EE \sum\limits_{(y_1,y_2,y_3) \in \eta^3}h(y_1,x)h(y_2,x)h(y_3,x)\\
			=~ & \EE \sum\limits_{(y_1,y_2,y_3) \in \eta^3_\neq}h(y_1,x)h(y_2,x)h(y_3,x) + 3\,\EE \sum\limits_{(y_1,y_2) \in \eta^2_\neq}h(y_1,x)^2h(y_2,x) + \EE \sum\limits_{y_1 \in \eta}h(y_1,x)^3\\
			=~& \int_{\RR^d} \int_{\RR^d} \int_{\RR^d}  h(y_1,x)h(y_2,x)h(y_3,x) \,\Lambda^3(\dint(y_1,y_2,y_3))\\
							& \qquad+ 3 \int_{\RR^d} \int_{\RR^d} h(y_1,x)^2h(y_2,x)\, \Lambda(\dint(y_1,y_2))+\int_{\RR^d} h(y_1,x)^3 \,\Lambda(\dint y_1)\,.
		\end{align*}
		Now, Fubini's theorem and the definition of $A_k$ imply (a).
		
		To prove part (b) we write
		\begin{align*}
			\EE[(D_xF_d)^4] = \EE \sum\limits_{(y_1,y_2,y_3,y_4) \in \eta^4}h(y_1,x)h(y_2,x)h(y_3,x)h(y_4,x)
		\end{align*}
		and split the sum in a similar way as above. Again, using the symmetry of $h$ and several times Mecke's formula \eqref{eq:Mecke}, the result follows.
		
		Finally, assertion (c) is clear from \eqref{eq:DxFh}.
	\end{proof}
\end{lemma}

We shall now provide the announced expressions for the terms $\g_1(\widetilde{F_d}),\g_2(\widetilde{F_d})$ and $\g_3(\widetilde{F_d})$.

\begin{lemma}\label{lemma:gamma_123}
We have that
	\begin{align*}
		\gamma_1(\widetilde{F_d}) & = \frac{1}{\sigma^4} \int_{\RR^d}\int_{\RR^d} B_{1,1}(x_1,x_2) \, \big( P(x_1) \, P(x_2) \big)^{1/4}\,\L^2(\dint(x_1,x_2))\,,\\
		\gamma_2(\widetilde{F_d}) & = \frac{1}{\sigma^4} \int_{\RR^d}\int_{\RR^d} B_{2,2}(x_1,x_2) \,\L(\dint x_1)\L(\dint x_2)\,,\\
		\gamma_3(\widetilde{F_d}) & = \frac{1}{\sigma^3} \int_{\RR^d} [A_1(x)^3 + 3 A_2(x)A_1(x) + A_3(x)] \,\L(\dint x)\,.
	\end{align*}
\end{lemma}	
	\begin{proof}
Using \eqref{eq:DxFhs}, we see that $\gamma_1(\widetilde{F_d})$ coincides with
		\begin{align*}
		\frac{1}{\sigma^4} \int_{(\RR^d)^3}\big(\EE[(D_{x_1}F_d)^4]\,\EE[(D_{x_2}F_d)^4]\, \EE[(D_{x_1,x_3}F_d)^4]\,\EE[(D_{x_2,x_3}F_d)^4]\big)^{1/4}\,\L^3(\dint(x_1,x_2,x_3))\,.
		\end{align*}
Using now Lemma \ref{lemma:gamma} (b) and (c), Fubini's theorem and the definition of the parameter integral $B_{1,1}$, we conclude that
		\begin{align*}
			\gamma_1(\widetilde{F_d})&=  \frac{1}{\sigma^4} \int_{(\RR^d)^3}\big( P(x_1) P(x_2) h(x_1,x_3)^4 h(x_2, x_3)^4\big)^{1/4}\L^3(\dint(x_1,x_2,x_3))\\
			&= \frac{1}{\sigma^4} \int_{\RR^d} \int_{\RR^d} \Big(\int_{\RR^d} h(x_1,x_3) h(x_2,x_3) \L(x_3) \Big)\, \big( P(x_1) P(x_2) \big)^{1/4}\, \L(x_1) \L(x_2)\\
			&=\frac{1}{\sigma^4} \int_{\RR^d}\int_{\RR^d} B_{1,1}(x_1,x_2) \, \big( P(x_1) \, P(x_2) \big)^{1/4}\,\L^2(\dint(x_1,x_2))\,.
		\end{align*}
Next, using Lemma \ref{lemma:gamma} (c), Fubini's theorem and the definition of $B_{2,2}$ we see that
		\begin{align*}
			\gamma_2(\widetilde{F_d}) & = \frac{1}{\sigma^4}\int_{(\RR^d)^3}\big(\EE[(D_{x_1,x_3}F_d)^4]\,\EE[(D_{x_2,x_3}F_d)^4]\big)^{1/2}\,\L^3(\dint(x_1,x_2,x_3))\\
				& = \frac{1}{\sigma^4}\int_{(\RR^d)^3}\big(h(x_1,x_3)^4 h(x_2,x_3)^4\big)^{1/2}\,\L^3(\dint(x_1,x_2,x_3))\\
				& = \frac{1}{\sigma^4}\int_{\RR^d} \int_{\RR^d} \Big( \int_{\RR^d} h(x_1,x_3)^2 h(x_2,x_3)^2\,\L(\dint x_3) \Big)\,\L(\dint x_1)\L(\dint x_2)\\
				& = \frac{1}{\sigma^4}\int_{\RR^d} \int_{\RR^d} B_{2,2}(x_1,x_2) \, \L(\dint x_1)\L(\dint x_2)\,,
		\end{align*}
		as desired. Finally, according to Lemma \ref{lemma:gamma} (a) we have that
				\begin{align*}
					\gamma_3(\widetilde{F_d}) &= \frac{1}{\sigma^3} \int_{\RR^d}\EE[|D_xF_d|^3]\,\L(\dint x) = \frac{1}{\sigma^3} \int_{\RR^d} [A_1(x)^3 + 3 A_2(x)A_1(x) + A_3(x)] \,\L(\dint x)
				\end{align*}
				and the proof is complete.
	\end{proof}

Now, we can combine these expressions established so far to reformulate Proposition \ref{prop:SecondOrderPoincare} for our second-order $U$-statistics $\widetilde{F_d}$.
\begin{proposition}\label{prop:generalCLT}
Let $\eta$ be a Poisson point process on $\RR^d$ with intensity measure $\Lambda$ and let $F_d=\sum_{(x,y)\in\eta_{\neq}^2}h(x,y)$ be a second-order $U$-statistic with a non-negative symmetric kernel $h$. Put $\sigma^2:=\Var[F_d]$, $\widetilde{F_d}:=\sigma^{-1}(F_d-\EE[F_d])$ and suppose that $		\sigma^{-2}\,\EE\int_{\RR^d}(D_xF_d)^2\,\L(\dint x) < \infty$. Defining
	\begin{align*}
		\gamma_1(\widetilde{F_d}) & := \frac{1}{\sigma^4} \int_{\RR^d}\int_{\RR^d} B_{1,1}(x_1,x_2) \, \big( P(x_1) \, P(x_2) \big)^{1/4}\,\L^2(\dint(x_1,x_2))\,,\\
		\gamma_2(\widetilde{F_d}) & := \frac{1}{\sigma^4} \int_{\RR^d}\int_{\RR^d} B_{2,2}(x_1,x_2) \,\L(\dint x_1)\L(\dint x_2)\,,\\
		\gamma_3(\widetilde{F_d}) & := \frac{1}{\sigma^3} \int_{\RR^d} [A_1(x)^3 + 3 A_2(x)A_1(x) + A_3(x)] \,\L(\dint x)\,,
	\end{align*}
	one has that
	\begin{align*}
		d_W(\widetilde{F_d},Z) \leq 2\sqrt{\g_1(\widetilde{F_d})} + \sqrt{\g_2(\widetilde{F_d})}+\g_3(\widetilde{F_d})\,,
	\end{align*}
	where $Z$ is a standard Gaussian random variable.
\end{proposition}

%
%
%
%
\section{Proof of Theorem \ref{thm:CLT}}\label{sec:Proof}

Let us recall that $\eta_d$ denotes a stationary Poisson point process on $\RR^d$ with intensity $\l_d$ given by \eqref{eq:DefLambdad}. We denote by $\L$ the intensity measure of $\eta_d$, that is, $\L$ is $\l_d$ times the Lebesgue measure on $\RR^d$. Moreover, from now on we will assume without loss of generality that all the random variables $(\cE_d)_{d\geq 2}$ are defined on a common probability space $(\Omega,\cF,\PP)$ and we denote expectation (integration) with respect to $\PP$ by $\EE$.

By definition of the edge counting statistic $\cE_d$ it is clear that $\cE_d$ is a second-order $U$-statistic with symmetric and $d$-dependent kernel
\begin{align}\label{eq:funcH}
	h(x_1,x_2) := {\bf 1}\Big\{\|x-y\|\leq\d_d,\frac{x+y}{2}\in\BB^d\Big\}\,.
\end{align}
Using \eqref{eq:ExpectationUstatistic} one has that
\begin{align}\label{eq:ExpectationRepresentation}
	\EE[\cE_d] = \frac{1}{2}\int_{\RR^d}A_1(x)\,\L(\dint x)= \frac{1}{2} \kappa_d^2 \lambda_d^2 \delta_d^d
\end{align}
and using \eqref{eq:UStatisticVariance} the identity
\begin{align}\label{eq:VarianceRepresentation}
	\sigma^2 := \Var[\cE_d] = \int_{\RR^d}A_1(x)^2\,\L(\dint x)+\frac{1}{2}\int_{\RR^d}A_2(x)\,\L(\dint x)
\end{align}
follows. To derive Theorem \ref{thm:CLT} we want to apply the normal approximation bound derived in Proposition \ref{prop:generalCLT}
 and for that purpose we first need to control the parameter integrals $A_k$.
\begin{lemma}
	Let $h:\RR^d \times \RR^d \rightarrow \{0,1\}$ be the function given by \eqref{eq:funcH}.
	Then, for all $x \in \RR^d$ and all $k \in \NN$ it holds that 
	\begin{align}\label{eq:parameter_integral:bound:A}
		{\bf 1}\Big\{ x \in \BB^d_{1-\frac{\delta_d}{2}}(0) \Big\}\, \kappa_d \lambda_d \delta_d^d \leq A_k(x) \leq {\bf 1}\Big\{ x \in \BB^d_{1+\frac{\delta_d}{2}}(0) \Big\} \,\kappa_d \lambda_d \delta_d^d\,.
	\end{align}
	\begin{proof}
		The function $h$ in \eqref{eq:funcH} takes only the values $0$ and $1$. Thus, $A_k(x) = A_1(x)$ and it is sufficient to show the bounds for the special choice $k = 1$.
		The parameter integral $A_1(x)$ can be re-written as
		\begin{align*}
			A_1(x) & =  \int_{\RR^d} {\bf 1}\Big\{\|x-y\|\leq\d_d,\frac{x+y}{2}\in\BB^d\Big\}\, \Lambda(\dint y)= \int_{\RR^d} {\bf 1}\Big\{y \in \BB_{\delta_d}^d(x) \cap \BB_2^d(-x) \Big\} \,\Lambda(\dint y) = \Lambda\left( \BB_{\delta_d}^d(x) \cap \BB_2^d(-x) \right)\,.
		\end{align*}
		Since $\L$ is just a multiple of the Lebesgue measure, it is clear that
		\begin{align*}
			\Lambda\left( \BB_{\delta_d}^d(x) \cap \BB_2^d(-x) \right) = 0 \quad \text{if and only if} \qquad
			 x \not\in \BB^d_{1+\frac{\delta_d}{2}}(0)\,.
		\end{align*}
		If the distance $\| x -(-x) \|$ of the two midpoints is smaller than the absolute value of the difference of the two radii, that is, smaller than $| 2 - \delta_d |$, then the smaller ball is contained in the larger one. In other words, $\BB_{\delta_d}^d(x) \subset \BB_2^d(-x)$ if and only if $x \in \BB^d_{1-\frac{\delta_d}{2}}(0)$ and we have that
		\begin{align*}
			\Lambda\left( \BB_{\delta_d}^d(x) \cap \BB_2^d(-x) \right) \geq \Lambda\left( \BB_{\delta_d}^d(x) \right)\,,
		\end{align*}
		provided $x \in \BB^d_{1 - \frac{\delta_d}{2}}(0)$. 
		We use the smaller ball $\BB_{\delta_d}^d(x)$ to give an upper bound for $A_1(x)$:
		\begin{align*}
			\Lambda\left( \BB_{\delta_d}^d(x) \cap \BB_2^d(-x) \right) \leq {\bf 1} \Big\{ x \in \BB^d_{1+\frac{\delta_d}{2}} \Big\}\, \kappa_d \lambda_d \delta_d^d\,, \qquad x \in \RR^d\,.
		\end{align*}
		If the smaller ball is a subset of the intersection we can use its volume as a lower bound for $A_1(x)$:
		\begin{align*}
			\Lambda\left( \BB_{\delta_d}^d(x) \cap \BB_2^d(-x) \right) \geq {\bf 1} \Big\{ x \in \BB^d_{1-\frac{\delta_d}{2}} \Big\}\, \kappa_d \lambda_d \delta_d^d\,, \qquad x \in \RR^d\,,
		\end{align*}
		completing thereby the proof.
	\end{proof}
\end{lemma}
In a next step, we shall derive a lower and an upper bound for the integral of $A_1(x)^2$ and hence for the variance $\sigma^2$.
\begin{lemma}\label{lem:VarianceBound}
It holds that
	\begin{align*}
		\Big( 1 - \tfrac{\delta_d}{2} \Big)^d \kappa_d^3 \lambda_d^3 \delta_d^{2d} \leq \int_{\RR^d} A_1(x)^2 \, \Lambda(\dint x) \leq \Big( 1 + \tfrac{\delta_d}{2} \Big)^d \kappa_d^3 \lambda_d^3 \delta_d^{2d}.
	\end{align*}
	In particular,
	\begin{align*}
		\frac{1}{2} \kappa_d^2 \lambda_d^2 \delta_d^d + \Big( 1 - \tfrac{\delta_d}{2} \Big)^d \kappa_d^3 \lambda_d^3 \delta_d^{2d} \leq \sigma^2 \leq \frac{1}{2} \kappa_d^2 \lambda_d^2 \delta_d^d + \Big( 1 + \tfrac{\delta_d}{2} \Big)^d \lambda_d^3 \kappa_d^3 \delta_d^{2d}.
	\end{align*}
\end{lemma}	
	\begin{proof}
		By \eqref{eq:parameter_integral:bound:A} it follows that
		\begin{align}\label{eq:A_bound}
			\int_{\RR^d} A_1(x)^2\,\L(\dint x) & \geq \int_{\RR^d} \left( {\bf 1}\Big\{ x \in \BB^d_{1-\frac{\delta_d}{2}} \Big\} \kappa_d \lambda_d \delta_d^d \right)^2 \Lambda(\dint x) = \Big(1 - \tfrac{\delta_d}{2} \Big)^d \kappa_d^3 \lambda_d^3 \delta_d^{2d},
		\end{align}
		and
		\begin{align}\label{eq:squared_A_bound}
			\int_{\RR^d} A_1(x)^2\,\L(\dint x) & \leq \int_{\RR^d} \left( {\bf 1}\Big\{ x \in \BB^d_{1+\frac{\delta_d}{2}} \Big\} \kappa_d \lambda_d \delta_d^d \right)^2 \Lambda(\dint x) = \Big(1 + \tfrac{\delta_d}{2} \Big)^d \kappa_d^3 \lambda_d^3 \delta_d^{2d}\,.
		\end{align}
		Moreover, using \eqref{eq:ExpectationRepresentation} and \eqref{eq:VarianceRepresentation} we have
		\begin{align*}
			\sigma^2 = \frac{1}{2} \kappa_d^2 \lambda_d^2 \delta_d^d + \int_{\RR^d} A_1(x)^2 \,\Lambda(\dint x)\,,
		\end{align*}
		which leads to the desired result.
	\end{proof}
\begin{remark}\label{remark:choose_delta}
	Our particular choice $\delta_d = \frac{1}{d}$ ensures that we can find absolute constants $c_1, c_2 \in (0, \infty)$ and $d_0 \in \NN$ such that $0 < c_1 \leq (1-\tfrac{\delta_d}{2})^d$ and $(1+\tfrac{\delta_d}{2})^d \leq c_2 < \infty$ for all $d \geq d_0$. The existence of such constants is important to derive the final bounds on the right hand side of our main result and implies restrictions to more general choices of $\delta_d$, see the proof of Lemma \ref{lemma:gamma_final}.
\end{remark}
In a next step, we shall check the integrability condition in Proposition \ref{prop:generalCLT}.
\begin{lemma}\label{lemma:integrability}
	It holds that
	\begin{align*}
		 \frac{1}{\sigma^2}\EE\int_{\RR^d}(D_x\cE_d)^2\,\L(\dint x) < \infty\,.
	\end{align*}
\end{lemma}	
	\begin{proof}
		It follows from \eqref{eq:DxFh} and Fubini's theorem that
		\begin{align*}
				\frac{1}{\sigma^2}\EE\int_{\RR^d}(D_x\cE_d)^2\,\L(\dint x)
			 = \frac{1}{\sigma^2}\int_{\RR^d} \EE\Big[\sum\limits_{(y_1,y_2) \in \eta^2}h(y_1,x)h(y_2,x)\Big]\, \L(\dint x)\,.
		\end{align*}
		Splitting the sum, using Mecke's formula \eqref{eq:Mecke} and applying \eqref{eq:ExpectationRepresentation} and \eqref{eq:VarianceRepresentation}, we conclude that the last expression is equal to
		\begin{align*}
			& \frac{1}{\sigma^2}\int_{\RR^d} \EE\Big[\sum\limits_{(y_1,y_2) \in \eta^2_{\neq}}h(y_1,x)h(y_2,x)\Big] + \EE\Big[\sum\limits_{y \in \eta} h(y,x)^2 \Big]\, \L(\dint x)\\
			&=  \frac{1}{\sigma^2}\int_{\RR^d} \int_{\RR^d}\int_{\RR^d} h(y_1,x)h(y_2,x) \,\L(\dint y_1)\L(\dint y_2)\L(\dint x) + \frac{1}{\sigma^2} \int_{\RR^d} \int_{\RR^d} h(y,x)^2\, \L(\dint y) \L(\dint x)\\
			&= \frac{1}{\sigma^2}\int_{\RR^d} A_1(x)^2\, \L(\dint x) + \frac{1}{\sigma^2} \int_{\RR^d} A_2(x) \,\L(\dint x)\\
			&=  \frac{1}{\sigma^2}\int_{\RR^d} A_1(x)^2 \,\L(\dint x) + \frac{1}{2\sigma^2} \int_{\RR^d} A_2(x) \,\L(\dint x) + \frac{1}{2\sigma^2} \int_{\RR^d} A_2(x) \,\L(\dint x)\\
			&=  \frac{\Var[\cE_d]}{\sigma^2} + \frac{\EE[\cE_d]}{\sigma^2}\,.
		\end{align*}
		Thus, using \eqref{eq:ExpectationRepresentation} and the lower variance bound from Lemma \ref{lem:VarianceBound}, one has that
		\begin{align*}
			\frac{1}{\sigma^2}\EE\int_{\RR^d}(D_x\cE_d)^2\,\L(\dint x) & = 1 + \frac{\EE[\cE_d]}{\Var[\cE_d]}\leq 1 + \frac{1}{1 + 2 (1-\tfrac{\delta_d}{2})^d \kappa_d \lambda_d \delta_d^d}\,,
		\end{align*}
	which is finite since $2(1-\tfrac{\delta_d}{2})^d \kappa_d \lambda_d \delta_d^d \geq 0$ for all $d \in \NN$.
	\end{proof}
Now, we will use the bounds for the parameter integrals $A_k$ to derive an upper bound for the three terms appearing in Proposition \ref{prop:generalCLT}.
\begin{lemma}\label{lem:boundsGamma123Final}
We have that
	\begin{align*}
		\gamma_1(\widetilde{\cE_d}) &\leq \frac{1}{\sigma^4} \kappa_d^3 \lambda_d^3 \delta_d^{2d} \Big( 1 + \tfrac{\delta_d}{2} \Big)^d \left[ (\kappa_d \lambda_d \delta_d^d)^4 + 6(\kappa_d \lambda_d \delta_d^d)^3 + 7(\kappa_d \lambda_d \delta_d^d)^2 + \kappa_d \lambda_d \delta_d^d \right]^\frac{1}{2}\,,\\
		\gamma_2(\widetilde{\cE_d}) &\leq \frac{1}{\sigma^4} \kappa_d^3 \lambda_d^3 \delta_d^{2d} \Big( 1 + \tfrac{\delta_d}{2} \Big)^d\,,\\
		\gamma_3(\widetilde{\cE_d}) &\leq  \frac{1}{\sigma^3} \kappa_d \lambda_d  \Big( 1 + \tfrac{\delta_d}{2} \Big)^d \left[ (\kappa_d \lambda_d \delta_d^d)^3 + 3(\kappa_d \lambda_d \delta_d^d)^2 + \kappa_d \lambda_d \delta_d^d \right]\,.
	\end{align*}
	\begin{proof}
		Applying \eqref{eq:parameter_integral:bound:A} to the definition of $P(x)$ in Lemma \ref{lemma:gamma}, we see that
		\begin{align*}
			P(x) & \leq {\bf 1}\Big\{ x \in \BB^d_{1+\frac{\delta_d}{2}}(0) \Big\} \left[ (\kappa_d \lambda_d \delta_d^d)^4 + 6(\kappa_d \lambda_d \delta_d^d)^3 + 7(\kappa_d \lambda_d \delta_d^d)^2 + \kappa_d \lambda_d \delta_d^d \right]\\
				 & \leq (\kappa_d \lambda_d \delta_d^d)^4 + 6(\kappa_d \lambda_d \delta_d^d)^3 + 7(\kappa_d \lambda_d \delta_d^d)^2 + \kappa_d \lambda_d \delta_d^d \,.
		\end{align*}
		Therefore, it follows that
		\begin{align*}
			\gamma_1(\widetilde{\cE_d}) & \leq \frac{1}{\sigma^4} \int_{\RR^d} \int_{\RR^d} B_{1,1}(x_1,x_2)\left[ (\kappa_d \lambda_d \delta_d^d)^4 + 6(\kappa_d \lambda_d \delta_d^d)^3 + 7(\kappa_d \lambda_d \delta_d^d)^2 + \kappa_d \lambda_d  \delta_d^d \right]^\frac{1}{2} \Lambda(\dint x_1) \Lambda(\dint x_2)\\
				& = \frac{1}{\sigma^4} \left[ (\kappa_d \lambda_d \delta_d^d)^4 + 6(\kappa_d \lambda_d \delta_d^d)^3 + 7(\kappa_d \lambda_d \delta_d^d)^2 + \kappa_d \lambda_d \delta_d^d \right]^\frac{1}{2}\int_{\RR^d} \int_{\RR^d} B_{1,1}(x_1,x_2) \Lambda(\dint x_1) \Lambda(\dint x_2).
		\end{align*}
		We now use Fubini's theorem to re-write the double integral. Together with \eqref{eq:squared_A_bound} this implies 
		\begin{align*}
			\gamma_1(\widetilde{\cE_d}) & \leq \frac{1}{\sigma^4} \left[ (\kappa_d \lambda_d \delta_d^d)^4 + 6(\kappa_d \lambda_d \delta_d^d)^3 + 7(\kappa_d \lambda_d  \delta_d^d)^2 + \kappa_d \lambda_d \delta_d^d \right]^\frac{1}{2}\int_{\RR^d} A_1(y)^2 \Lambda(\dint y)\\
				& \leq \frac{1}{\sigma^4} \kappa_d^3 \lambda_d^3 \delta_d^{2d} \Big( 1 + \tfrac{\delta_d}{2} \Big)^d \left[ (\kappa_d \lambda_d \delta_d^d)^4 + 6(\kappa_d \lambda_d \delta_d^d)^3 + 7(\kappa_d \lambda_d \delta_d^d)^2 + \kappa_d \lambda_d \delta_d^d \right]^\frac{1}{2}\,.
		\end{align*}

Using \eqref{eq:squared_A_bound} we obtain in a similar way that
		\begin{align*}
			\gamma_2(\widetilde{\cE_d}) = \frac{1}{\sigma^4} \int_{\RR^d}\int_{\RR^d} B_{2,2}(x_1,x_2) \,\L(\dint x_1)\L(\dint x_2) \leq \frac{1}{\sigma^4} \int_{\RR^d} A_1(y)^2 \,\L(\dint y) \leq \frac{1}{\sigma^4} \Big(1 + \tfrac{\delta_d^d}{2} \Big)^d \kappa_d^3 \lambda_d^3 \delta_d^{2d}\,.
		\end{align*}		
		
Finally, using \eqref{eq:parameter_integral:bound:A} it follows that 
		\begin{align*}
			\gamma_3(\widetilde{\cE_d}) & \leq \frac{1}{\sigma^3} \int_{\RR^d} {\bf 1}\Big\{ x \in \BB^d_{1+\frac{\delta_d}{2}}(0) \Big\} \left[ (\kappa_d \lambda_d \delta_d^d)^3 + 3(\kappa_d \lambda_d \delta_d^d)^2 + \kappa_d \lambda_d \delta_d^d \right] \Lambda(\dint x)\\
										& = \frac{1}{\sigma^3} \kappa_d \lambda_d \Big( 1 + \tfrac{\delta_d}{2} \Big)^d \left[ (\kappa_d \lambda_d \delta_d^d)^3 + 3(\kappa_d \lambda_d \delta_d^d)^2 + \kappa_d \lambda_d \delta_d^d \right]
		\end{align*}		
		and the proof is complete.
	\end{proof}
\end{lemma}
To derive the asymptotic behaviour of the three bounds in the previous lemma, we will distinguish the following regimes:
\begin{align}\label{eq:cases:1}
	\lim_{d\to\infty}\kappa_d^2\l_d^2\d_d^{2d} & = \infty\,,\\
	\label{eq:cases:2}
	\lim_{d\to\infty}\kappa_d^2\l_d^2\d_d^{2d} & = c^2\in(0,\infty)\,,\\
	\label{eq:cases:3}
	\lim_{d\to\infty}\kappa_d^2\l_d^2\d_d^{2d} & = 0\,.
\end{align}
In the next lemma we shall provide upper bounds for $\gamma_1(\widetilde{\cE_d})$, $\gamma_2(\widetilde{\cE_d})$ and $\gamma_3(\widetilde{\cE_d})$ as well as lower bounds for the variance $\sigma^2$ for each of these regimes.
\begin{lemma}\label{lemma:gamma_final}
	In the first regime \eqref{eq:cases:1} there exists absolute constants $c_1, c_2, c_3 \in (0,\infty)$ and $d_1 \in \NN$ such that
	\begin{alignat*}{2}
		\sigma^2 & \geq (1-\tfrac{\delta_d}{2})^d \kappa_d^3 \lambda_d^3 \delta_d^{2d}\,,\qquad\qquad\qquad\qquad
		&&\gamma_1(\widetilde{\cE_d}) \leq c_1(\kappa_d \lambda_d)^{-1}\,,\\
		\gamma_2(\widetilde{\cE_d}) & \leq c_2(\kappa_d \lambda_d)^{-1}(\kappa_d \lambda_d \delta_d^d)^{-2}\,,\qquad
		&&\gamma_3(\widetilde{\cE_d}) \leq c_3(\kappa_d \lambda_d)^{-\frac{1}{2}}\,.
	\end{alignat*}
	for all $d \geq d_1$.
	In the second regime \eqref{eq:cases:2}, we have absolute constants $c_4, c_5, c_6 \in (0,\infty)$ and $d_2 \in \NN$ such that
	\begin{alignat*}{2}
		\sigma^2 & \geq \frac{1}{2} \kappa_d^2 \lambda_d^2 \delta_d^d\,,\qquad\qquad\qquad\qquad\qquad
		&&\gamma_1(\widetilde{\cE_d}) \leq c_4 (\kappa_d \lambda_d)^{-1}\,,\\
		\gamma_2(\widetilde{\cE_d}) & \leq c_5 (\kappa_d \lambda_d)^{-1}\,,
		&&\gamma_3(\widetilde{\cE_d}) \leq c_6 (\kappa_d \lambda_d)^{-\frac{1}{2}}\,,
	\end{alignat*}
	holds for all $d\geq d_2$, while in the third regime \eqref{eq:cases:3} we can find absolute constants $c_7, c_8, c_9 \in (0,\infty)$ and $d_3 \in \NN$ such that
	\begin{alignat*}{2}
		\sigma^2 & \geq \frac{1}{2} \kappa_d^2 \lambda_d^2 \delta_d^d\,,\qquad\qquad\qquad\qquad\qquad
		&&\gamma_1(\widetilde{\cE_d}) \leq c_7 (\kappa_d \lambda_d)^{-\frac{1}{2}}\delta_d^{\frac{d}{2}}\,,\\
		\gamma_2(\widetilde{\cE_d}) & \leq c_8 (\kappa_d \lambda_d)^{-1}\,,
		&&\gamma_3(\widetilde{\cE_d}) \leq c_9 (\kappa_d \lambda_d)^{-\frac{1}{2}}(\kappa_d \lambda_d \delta_d^d)^{-\frac{1}{2}}
	\end{alignat*}
	if $d\geq d_3$.
	\begin{proof}
		Using Lemma \ref{lem:VarianceBound} one can directly obtain the variance estimates by omitting the first or second term in the sum, respectively.
		Together with Lemma \ref{lem:boundsGamma123Final} it follows in the first regime \eqref{eq:cases:1} that
		\begin{align*}
			\gamma_1(\widetilde{\cE_d}) & \leq \big( (1-\tfrac{\delta_d}{2})^d \kappa_d^3 \lambda_d^3 \delta_d^{2d} \big)^{-2} \kappa_d^3 \lambda_d^3 \delta_d^{2d} (1+\tfrac{\delta_d}{2})^d \,\big[ (\kappa_d\lambda_d\delta_d^d)^4 + 6(\kappa_d\lambda_d\delta_d^d)^3 + 7(\kappa_d\lambda_d\delta_d^d)^2 + (\kappa_d\lambda_d\delta_d^d)^1 \big]^{\tfrac{1}{2}}\,.
		\end{align*}
		As explained in Remark \ref{remark:choose_delta} we can find absolute constants $\tilde{c}_1, \tilde{c}_2 \in (0,\infty)$ and $\tilde{d}_0 \in \NN$ such that
		\begin{align*}
			\gamma_1(\widetilde{\cE_d}) & \leq \tilde{c}_1^{-2} \tilde{c}_2  \kappa_d^{-3} \lambda_d^{-3} \delta_d^{-2d} \big[ (\kappa_d\lambda_d\delta_d^d)^4 + 6(\kappa_d\lambda_d\delta_d^d)^3 + 7(\kappa_d\lambda_d\delta_d^d)^2 + (\kappa_d\lambda_d\delta_d^d)^1 \big]^{\tfrac{1}{2}}
		\end{align*}
		holds for all $d \geq \tilde{d}_0$.
		Since $\lim\limits_{d\to\infty}\kappa_d\l_d\d_d^d = \infty$ we can find further absolute constants $\tilde{c}_3 \in (0, \infty)$ and $\tilde{d}_1 \in \NN$ such that
		\begin{align*}
			\gamma_1(\widetilde{\cE_d}) & \leq \tilde{c}_1^{-2} \tilde{c}_2 \tilde{c}_3 \kappa_d^{-3} \lambda_d^{-3} \delta_d^{-2d} (\kappa_d\lambda_d\delta_d^d)^{\tfrac{4}{2}},
		\end{align*}
		holds for all $d \geq \max(\tilde{d}_0, \tilde{d}_1)$. This directly leads to the desired bound for $\gamma_1(\widetilde{\cE_d})$.
		Using Lemma \ref{lem:boundsGamma123Final} and Remark \ref{remark:choose_delta} we further find that the inequality
		\begin{align*}
			\gamma_2(\widetilde{\cE_d}) & \leq \big( (1-\tfrac{\delta_d}{2})^d \kappa_d^3 \lambda_d^3 \delta_d^{2d} \big)^{-2} \kappa_d^3 \lambda_d^3 \delta_d^{2d} (1+\tfrac{\delta_d}{2})^d
										\leq \tilde{c}_1^{-2} \tilde{c}_2 \kappa_d^{-3} \lambda_d^{-3} \delta_d^{-2d}
										= \tilde{c}_1^{-2} \tilde{c}_2 (\kappa_d \lambda_d)^{-1} (\kappa_d \lambda_d \delta_d^d)^{-2}
		\end{align*}
		holds for all sufficiently large space dimensions $d$. Finally, for $\gamma_3(\widetilde{\cE_d})$ we get
		\begin{align*}
			\gamma_3(\widetilde{\cE_d}) & \leq \big( (1-\tfrac{\delta_d}{2})^d \kappa_d^3 \lambda_d^3 \delta_d^{2d} \big)^{-\frac{3}{2}} \kappa_d \lambda_d (1+\tfrac{\delta_d}{2})^d \left[ (\kappa_d \lambda_d \delta_d^d)^3 + 3(\kappa_d \lambda_d \delta_d^d)^2 + \kappa_d \lambda_d \delta_d^d \right]\,.
		\end{align*}
		Again, Remark \ref{remark:choose_delta} implies that there are absolute constants $\tilde{c}_1, \tilde{c}_2 \in (0,\infty)$ and $\tilde{d}_0 \in \NN$ such that
		\begin{align*}
			\gamma_3(\widetilde{\cE_d}) & \leq \tilde{c}_1^{-\frac{3}{2}}\tilde{c}_2 \kappa_d^{-\frac{9}{2}} \lambda_d^{-\frac{9}{2}} \delta_d^{-3d} \kappa_d \lambda_d \left[ (\kappa_d \lambda_d \delta_d^d)^3 + 3(\kappa_d \lambda_d \delta_d^d)^2 + \kappa_d \lambda_d \delta_d^d \right]
		\end{align*}
		holds for all $d \geq \tilde{d}_0$.
		Since $\lim\limits_{d\to\infty}\kappa_d\l_d\d_d^d = \infty$ we can find further absolute constants $\tilde{c}_4 \in (0, \infty)$ and $\tilde{d}_2 \in \NN$ such that
		\begin{align*}
			\gamma_3(\widetilde{\cE_d}) & \leq \tilde{c}_1^{-\frac{3}{2}}\tilde{c}_2\tilde{c}_4 \kappa_d^{-\frac{9}{2}} \lambda_d^{-\frac{9}{2}} \delta_d^{-3d} \kappa_d \lambda_d (\kappa_d \lambda_d \delta_d^d)^3
		\end{align*}
		is valid, provided that $d \geq \max(\tilde{d}_0, \tilde{d}_2)$. This proves the bound for $\gamma_3(\widetilde{\cE_d})$ in regime \eqref{eq:cases:1}. Since the estimates in regimes \eqref{eq:cases:2} and \eqref{eq:cases:3} follow in a similar way, we omit the details.
	\end{proof}
\end{lemma}
After these preparations, we can now present the proof of our main result.

\begin{proof}[Proof of Theorem \ref{thm:CLT}]
We use Proposition \ref{prop:generalCLT} and apply the results of the last lemma. In the first regime \eqref{eq:cases:1}, we get absolute constants $c_1, c_2, c_3 \in (0, \infty)$ and $d_1 \in \NN$ such that
\begin{align*}
	d_W(\widetilde{\cE_d},Z) & \leq 2 \sqrt{c_1} \sqrt{(\kappa_d \lambda_d)^{-1}} + \sqrt{c_2} \sqrt{(\kappa_d \lambda_d)^{-1}(\kappa_d \lambda_d \delta_d^d)^{-2}}+ c_3 (\kappa_d \lambda_d)^{-\frac{1}{2}}\\
		& \leq 2 \sqrt{c_1} (\kappa_d \lambda_d)^{-\frac{1}{2}} + \sqrt{c_2} (\kappa_d \lambda_d)^{-\frac{1}{2}}(\kappa_d \lambda_d \delta_d^d)^{-1} + c_3 (\kappa_d \lambda_d)^{-\frac{1}{2}}\\
		& \leq (2 \sqrt{c_1} + c_3) (\kappa_d \lambda_d)^{-\frac{1}{2}} + \sqrt{c_2} (\kappa_d \lambda_d)^{-\frac{1}{2}}(\kappa_d \lambda_d \delta_d^d)^{-1}
\end{align*}
holds for all $d > d_1$. Since in the first regime, $\kappa_d\l_d\d_d^d \rightarrow \infty$, which implies that $(\kappa_d \lambda_d \delta_d^d)^{-1} \rightarrow 0$, we see that $(\kappa_d \lambda_d)^{-\frac{1}{2}}$ is the asymptotically leading term. It thus follows that there exist absolute constants $\bar{c}_1 \in (0, \infty)$ and $\bar{d}_1 \in \NN$ such that
\begin{align*}
	d_W(\widetilde{\cE_d},Z) & \leq \bar{c}_1 (\kappa_d \lambda_d)^{-\frac{1}{2}}
\end{align*}
for all $d > \bar{d}_1$.

In the second regime \eqref{eq:cases:2}, we get absolute constants $c_4, c_5, c_6 \in (0, \infty)$ and $d_2 \in \NN$ such that 
\begin{align*}
	d_W(\widetilde{\cE_d},Z) & \leq 2 \sqrt{c_4} \sqrt{(\kappa_d \lambda_d)^{-1}} + \sqrt{c_5} \sqrt{(\kappa_d \lambda_d)^{-1}} + c_6 (\kappa_d \lambda_d)^{-\frac{1}{2}}\\
		& \leq 2 \sqrt{c_4} (\kappa_d \lambda_d)^{-\frac{1}{2}} + \sqrt{c_5} (\kappa_d \lambda_d)^{-\frac{1}{2}} + c_6 (\kappa_d \lambda_d)^{-\frac{1}{2}}\\
		& \leq (2 \sqrt{c_4} + \sqrt{c_5} + c_6) (\kappa_d \lambda_d)^{-\frac{1}{2}}
\end{align*}
for all $d > d_2$, which directly leads to absolute constants $\bar{c}_2 := (2 \sqrt{c_4} + \sqrt{c_5} + c_6)$ and $\bar{d}_2 := d_2$ such that
\begin{align*}
	d_W(\widetilde{\cE_d},Z) & \leq \bar{c}_2 (\kappa_d \lambda_d)^{-\frac{1}{2}},
\end{align*}
for all $d > \bar{d}_2$.

Finally, in the third regime \eqref{eq:cases:3}, we get absolute constants $c_7, c_8, c_9 \in (0, \infty)$ and $d_3 \in \NN$ such that
\begin{align*}
	d_W(\widetilde{\cE_d},Z) & \leq 2 \sqrt{c_7} \sqrt{(\kappa_d \lambda_d)^{-\frac{1}{2}}\delta_d^{\frac{d}{2}}} + \sqrt{c_8} \sqrt{(\kappa_d \lambda_d)^{-1}} + c_9 (\kappa_d \lambda_d)^{-\frac{1}{2}}(\kappa_d \lambda_d \delta_d^d)^{-\frac{1}{2}}\\
		& \leq 2 \sqrt{c_7} (\kappa_d \lambda_d)^{-\frac{1}{4}}\delta_d^{\frac{d}{4}} + \sqrt{c_8} (\kappa_d \lambda_d)^{-\frac{1}{2}} + c_9 ( \kappa_d \lambda_d)^{-\frac{1}{2}}(\kappa_d \lambda_d \delta_d^d)^{-\frac{1}{2}}
\end{align*}
for all $d > d_3$. Since $(\kappa_d \lambda_d)^{-\frac{1}{4}}\delta_d^{\frac{d}{4}} = (\kappa_d \lambda_d)^{-\frac{1}{2}} (\kappa_d \lambda_d \delta_d^d)^\frac{1}{4}$, we see by \eqref{eq:cases:3} that the first term in the sum tends to zero faster than the second term. Further, it follows from the fact that $(\kappa_d \lambda_d \delta_d^d)^{-\frac{1}{2}} \rightarrow \infty$ that the third term in the sum tends to zero slower than the second. This allows us to find absolute constants $\bar{c}_3 \in (0, \infty)$ and $\bar{d}_3 \in \NN$ such that
\begin{align*}
	d_W(\widetilde{\cE_d},Z) & \leq \bar{c}_3 (\kappa_d \lambda_d)^{-\frac{1}{2}}(\kappa_d \lambda_d \delta_d^d)^{-\frac{1}{2}}
\end{align*}
holds for all $d > \bar{d}_3$.

Combining these three cases leads to absolute constants $c \in (0, \infty)$ and $d_0 \in \NN$ such that
\begin{align*}
	d_W(\widetilde{\cE_d},Z) \leq c (\kappa_d \lambda_d)^{-\frac{1}{2}} \max\Big\{1, (\kappa_d\lambda_d \delta_d^d)^{-\frac{1}{2}} \Big\}
\end{align*}
holds for all $d > d_0$. Using our assumption \eqref{eq:DefLambdad} it follows that $d_W(\widetilde{\cE_d},Z) \rightarrow 0$ and hence $\widetilde{\cE_d} \overset{D}{\longrightarrow} Z$, as $d\to\infty$. This completes the proof of Theorem \ref{thm:CLT}.
\end{proof}


\end{document}